\documentclass[11pt]{article}
\usepackage{amssymb}
\usepackage{graphicx}
\usepackage{xcolor} 
\usepackage{tensor}
\usepackage{fullpage} 
\usepackage{amsmath}
\usepackage{amsthm}
\usepackage{verbatim}
\usepackage{hyperref}
\usepackage{cite}
\usepackage{array} 
\usepackage[UKenglish,USenglish]{babel}
\usepackage{csquotes}
\usepackage{enumitem}
\usepackage{footnote}
\setlist[enumerate]{leftmargin=1.2em}
\setlist[itemize]{leftmargin=1.2em}
\usepackage{seqsplit,mathtools}

\definecolor{green}{rgb}{0,0.8,0} 

\newtheorem{theorem}{Theorem}[section]
\newtheorem{corollary}[theorem]{Corollary}
\newtheorem{lemma}[theorem]{Lemma}

\newtheorem{hypothesis*}{Hypothesis}

\theoremstyle{definition}

\theoremstyle{remark}
\newtheorem{remark}[theorem]{Remark}

\numberwithin{equation}{section}

\newcommand{\nnrm}[1]{{\vert\kern-0.25ex\vert\kern-0.25ex\vert #1 
		\vert\kern-0.25ex\vert\kern-0.25ex\vert}}

\title{Finite-time singularity formation for scalar stretching equations}
\author{Roberta Bianchini\footnote{Consiglio Nazionale delle Ricerche CNR-IAC, Rome, Italy. E-mail: roberta.bianchini@cnr.it}\, and Tarek M. Elgindi\footnote{Department of Mathematics, Duke University, Durham, NC. E-mail: tarek.elgindi@duke.edu}}
\date{\today}

\begin{document}

\maketitle
\begin{abstract}
We consider equations of the type:
\[\partial_t \omega = \omega R(\omega),\] for general linear operators $R$ in any spatial dimension. We prove that such equations almost always exhibit finite-time singularities for smooth and localized solutions. Singularities can even form in settings where solutions dissipate an energy. Such equations arise naturally as models in various physical settings such as inviscid and complex fluids. 
\end{abstract}

\setcounter{tocdepth}{2}
\tableofcontents

\section{Introduction}

Singularity formation in non-linear PDEs is the source of a number of interesting phenomena. In general, we would like to know what are the main mechanisms that lead to singularity formation. The purpose of this work is to show that singularities are inherent to a certain type of equation and that they appear whenever possible. In particular, imposing energetic constraints (as long as they do not \emph{a-priori} exclude singularity formation) is not sufficient to prevent singularity formation. Specifically, we consider solutions to scalar evolution equations of the form:
\[\partial_t \omega = \omega R(\omega),\] where $R$ is a general linear operator.  Note that this includes a number of well-known equations like the Burgers equation, some thin film equations, crystal growth models \cite{marzuola2013relaxation}, models of vorticity stretching in incompressible fluids \cite{CLM,miller2023finite}, and models of complex fluids \cite{constantin2012}.

\subsection{Organization of the Paper}

Section \ref{GeneralCase} gives a general criterion, Theorem \ref{GeneralTheorem}, that implies finite-time singularity formation for any equation for the form \eqref{eq:model} with $R$ a general linear operator. We then mention a few applications of this criterion. Section \ref{Energy} deals with a particular setting that is not easily recovered from the general Theorem \ref{GeneralTheorem} where solutions dissipate an energy and the main result there is Theorem \ref{EnergyTheorem}. A slightly counter-intuitive Theorem \ref{HALemma} of possibly independent interest, relating the positivity of a function and its Riesz transform, is established in the course of the proof of Theorem \ref{EnergyTheorem}. 

\section{General Case}\label{GeneralCase}
This section concerns solutions to the general equation:
\begin{equation}\label{eq:model}
\begin{cases}
    \partial_t \omega = \omega R(\omega),\\
    \omega (0, x) = \omega_0(x),  
\end{cases}
\end{equation}
on $\mathbb{R}^d$ for some $d\geq 1$, where $R$ is a (possibly unbounded) linear operator\footnote{We could think of $R$ as a Fourier multiplier.} on $L^2$. Let $R^*$ be the formal adjoint of $R$.  Let us make the following hypothesis on $R$  and $\omega_0$:

\begin{hypothesis*}
    Assume that there exist:
    \begin{itemize}
    \item a non-negative $W_2\in L^1(\mathbb{R}^d);$ 
    \item an $\omega_0$ for which $R^*(W_2) \omega_0\geq 0$
    \end{itemize}
    satisfying 
    \[-\infty<\int \log\Big(\frac{\omega_0R^*(W_2)}{W_2}\Big) W_2,\qquad 0<\int\omega_0 R^*(W_2)<\infty.\]
\end{hypothesis*}
Note that we are implicitly assuming in the Hypothesis that $R^*(W_2)$ is well-defined.   It is easy to see that most linear operators that arise in applications satisfy this hypothesis (see also Section \ref{Examples}). 
We now have the following general theorem. 
\begin{theorem}\label{GeneralTheorem}
Assume that $\omega_0, R$ satisfy the Hypothesis. Then, any smooth solution to \eqref{eq:model} must develop a singularity in finite time. 
\end{theorem}

\begin{proof}
Set $W_1=R^*(W_2)$ and assume without loss of generality that $\int_{\mathbb{R}^d}W_2=1.$
   We have that \[\omega=\omega_0 \exp\Big (\int_0^t R(\omega) \Big).\]
   Let us integrate against $W_1$. Then we get:
   \[\int \omega W_1 = \int \omega_0 \exp\Big (\int_0^t R(\omega) \Big)W_1=\int \exp\Big(\log\Big(\frac{\omega_0W_1}{W_2}\Big)+\int_0^t R(\omega) \Big) W_2.\] 
   Now we apply Jensen's inequality with the measure $W_2 dx$ and get:
   \[\int \omega W_1\geq c_*\exp\Big(\int_0^t\int R(\omega) W_2\Big)=c_*\exp\Big(\int_0^t\int \omega W_1\Big),\] where we used that $R^*(W_2)=W_1$ in the final equality. 
   It follows that $\omega$ must develop a singularity in finite time. 
\end{proof}

Let us mention a first interesting corollary:
\begin{corollary}
Consider \eqref{eq:model} posed on $\mathbb{T}^d.$ Assume that $R$ is a (possibly unbounded) linear operator on $L^2(\mathbb{T}^d)$ for which $R, R^*$ map analytic functions on $\mathbb{T}^d$ to analytic functions on $\mathbb{T}^d$. Either $R\equiv 0$ or there exists an analytic $\omega_0$ so that any solution to \eqref{eq:model} must develop a singularity in finite time. 
\end{corollary}
\begin{remark}
For general $R,$ this equation might not even be locally solvable in the space of analytic functions. The case of solutions posed on $\mathbb{R}^d$ is similar. 
\end{remark}
\begin{proof}
Consider $R^*(1)$. If it is not identically zero, then we may take $\omega_0=R^*(1)$, and $\log(R^*(1)^2)$ is integrable \cite{GG}. Thus, the condition is satisfied and we have a singularity. Otherwise $R^*(1)=0$. If $R^*\not\equiv0$, then we fix an analytic $\tilde W_2$ with $R^*(\tilde W_2)\not\equiv 0$ and define $W_2=M+\tilde W_2$ for $M$ large so that $W_2> 0.$ Then we take $\omega_0=W_2 R^*(W_2)$ and note that $\log (R^*(W_2)^2)$ is integrable. Thus the hypothesis is satisfied and we have a singularity. 
\end{proof}

\subsection{Some Examples}\label{Examples}

Let us give a few examples. 

\subsubsection{The Burgers Equation}
This is the case where $R=\partial_x$ on $\mathbb{R}$ or $\mathbb{T}.$ In this case, we may take $W_2(x)=\frac{1}{(1+x^2)^2}$ and then $W_1(x)=-W_2'(x)=\frac{4x}{(1+x^2)^3}.$ Then we see the condition implies that $\omega_0$ should satisfy:
\[\Big|\int_{-\infty}^\infty \log \left(\frac{4x\omega_0(x)}{1+x^2}\right) \frac{1}{(1+x^2)^2}\Big|<\infty.\] Taking $\omega_0=W_1(x)\exp(-x^2)$ does the job. Note, as a sanity check, that the integrability condition implies that $\omega_0$ must be non-positive on $(-\infty,0]$ and non-negative on $[0,\infty)$. This implies that $\omega_0'(x)>0$ for some $x.$ This is consistent with the classical fact that solutions to the Burgers equation (with a minus sign!) become singular if and only if there exists a point where the derivative of the data is positive. 

\subsubsection{The Constantin-Lax-Majda Equation}
This is the case where $R=H$ is the Hilbert transform. The problem can be posed on $\mathbb{R}$ or on $\mathbb{T}$. The classical proofs of singularity formation for the Constantin-Lax-Majda equation all rely on \emph{non-linear} identities related to the Hilbert transform. By taking $W_2(x)=\frac{1}{1+x^2}$, we see that $H(W_2)=c \frac{x}{1+x^2}.$ Consequently, we deduce singularity formation for \emph{any} $\omega_0$ for which 
\[\Big|\int \log (cx\omega_0(x)) \frac{1}{1+x^2}\Big|<\infty.\] Note, as another sanity check, that the condition implies that $c x\omega_0(x)\geq 0,$ which in particular implies that $\omega_0$ vanishes at $x=0$ and that its Hilbert transform is positive at $x=0$ (this is consistent with the result of \cite{CLM}). On $\mathbb{T}^2,$ we similarly note that we may take $W_2(x)=1+\cos(x)$ and $W_1=-H(W_2)=-\sin(x).$ We thus see that the condition on $\omega_0$ becomes
\[\int \log (-\sin(x)\omega_0(x))(1+\cos(x))dx>-\infty,\] which is manifestly true, for example, when $\omega_0(x)=-\sin(x).$ 

\subsubsection{A 2d Equation}
The next example we give resolves a question raised by Kiselev in \cite{Kiselev}. Here we take $R=R_{12}$, the composition of the Riesz transforms $R_1$ and $R_2$. Whether a finite-time singularity for this particular equation occurs was raised also in \cite{E_Classical}, where it was remarked that the $C^\alpha$ theory developed there implies that $C^\alpha$ solutions can develop a singularity in finite time. Using the above theorem, it is not difficult to exhibit a \emph{smooth} and rapidly decaying initial datum for which the unique local solution develops a singularity. Indeed, we may take $W_2(x)=\frac{1}{(1+|x|^2)^3}$ and it is not difficult to show that the data $\omega_0(x_1,x_2)= R_{12}(W_2) \exp(-|x|^2)$ satisfies the Hypothesis. Similarly, on $\mathbb{T}^2,$ we may take $W_2=1+\cos(x)\cos(y).$ Then, $R_{12}(W_2)=\sin(x)\sin(y),$ so that we may take $\omega_0(x,y)=\sin(x)\sin(y)$ and get a singularity. 

\subsubsection{Vortex stretching in swirl-free solutions to the 3d Euler equation}
When modeling just the effect of vortex stretching in the swirl-free axi-symmetric Euler equation, $R$ is taken to be the linear map that takes $\omega^\theta\rightarrow \frac{u_r}{r},$ where \[u_r(r,z)=\int_{-\infty}^\infty\int_{0}^\infty K(r,r',z,z')\omega^\theta(r',z')dr'dz',\] for some explicit kernel $K$ (see, for example, \cite{miller2023finite} or \cite{feng2015cauchy}). 
Taking $W_2(r,z)=\frac{1}{(1+r^2+z^2)^3}$ and setting $\omega_0^\theta=R^*(W_2)\exp(-(r^2+z^2))$ gives us a singularity. This recovers the result of \cite{miller2023finite}.

\subsubsection{A non-example when we put further conditions on $\omega_0$}
In the case where $R=-Id$ and we consider \emph{non-negative} initial data, it is obviously \emph{impossible} to find $\omega_0$ that is non-negative for which $\omega_0 R^*(W_2)\geq 0.$ This is consistent with the fact that non-negative solutions in this case always satisfy $0\leq \omega\leq \omega_0.$ This simple example is meant to motivate the coming section. 

\section{An equation with an energy}\label{Energy}
Let us now consider a slightly more complicated setting, which is related to the non-example above.  Consider non-negative solutions to the following model on $\mathbb{R}^2:$
\begin{equation}\label{eq:CSmodel}
\begin{cases}
    \partial_t \omega = \omega R_1^2\omega,\\
    \omega (0, x) = \omega_0(x).  
\end{cases}
\end{equation} Here, $R_1$ is the first component of the Riesz transform. Since $R_1^2$ is a negative operator on $L^2,$ we get the following identities:
\begin{align*}
    \frac{d}{dt}|\omega|_{L^1}&=-|R_1\omega|_{L^2}^2,\\
    \frac{d}{dt}|R_1\omega|_{L^2}^2&=-2\int\omega (R_1^2\omega)^2\leq 0.
\end{align*}
Using these dissipative properties, Constantin and Sun \cite{constantin2012} deduced global regularity for a wide class of non-negative solutions (in fact, for a wider class of equations when $R_1$ is replaced by more general anti-symmetric operators). Equation \eqref{eq:CSmodel} can be seen as a toy model of the Oldroyd B system \cite{constantin2012}. 

We now show that general smooth non-negative solutions to \eqref{eq:CSmodel} can develop singularities in finite time, answering a question raised in \cite{constantin2012}.

\begin{theorem}\label{EnergyTheorem}
There exists a non-negative, smooth, and compactly supported $\omega_0$ so that the unique solution to \eqref{eq:CSmodel} develops a singularity in finite time. 
\end{theorem}
The proof given in Section \ref{GeneralCase} does not seem to carry over here because we are searching for non-negative solutions and $R_1^2$ is a negative operator. While similar in spirit, the proof we will give here is a bit more involved and it will rely on a number of observations specific to the operator $R_1^2$. A particularly important observation is that for a specific type of data $\omega_0$, which is highly concentrated in a certain way around the origin and supported in a particular conical region, $R_1^2\omega_0$ restricted to a small ball around the origin is large and non-negative. A version of the leading order expansions given in \cite{E_Classical} is used to establish this. In fact, it can be seen from the proof that we actually establish the following.
\begin{theorem}\label{HALemma}
Consider $R_1^2$, the composition of the first component of the Riesz transform with itself. There exists a non-trivial and non-negative $\tilde{W}\in L^1(\mathbb{R}^2)$ for which $R_1^2(\tilde W)$ is non-negative on the support of $\tilde W.$ 
\end{theorem}
This lemma may seem to contradict the fact that $R_1^2$ is a negative operator on $L^2$, since we then should have that $(R_1^2(f),f)_{L^2}\leq 0;$ the point is that $\tilde W\not\in L^2.$ Since $\tilde W$ is algebraically unbounded and in $L^1$, it actually belongs to $L^p$ for some $1<p<2.$ We are not aware of such a phenomenon being studied in the harmonic analysis literature, and it may be interesting to study further questions related to the signs of $W$ and $R(W)$ under various assumptions on $W$ and $R$. 

To construct such a weight $\tilde W,$ we have to carefully design its angular and radial dependence.  The two key points are that $\tilde{W}$ is sufficiently unbounded as $r\rightarrow 0$ and it is supported in a cone about the angle $\theta=\frac{\pi}{2}$ (the reason for this latter choice has to do with the nature of $R_1^2$). Here, $(r,\theta)$ are the usual polar radius and angle on $\mathbb{R}^2.$

\subsection{Angular decomposition and localization}
We now turn to give the proof of Theorem \ref{EnergyTheorem}.
For $k\in \mathbb{N}$ and $f:[0,\infty)\times\mathbb{S}^1\rightarrow\mathbb{R},$ we write:
\begin{equation}
\begin{aligned}
f_0(r):&=\frac{1}{2\pi}\int_{-\pi}^\pi f(r,\theta)\, d\theta, \\
f_k(r):&=\frac{1}{\pi}\int_{-\pi}^\pi f(r,\theta)\cos(k\theta)\, d\theta, \quad k \ge 1.
\end{aligned} \label{eq:kmode}
\end{equation}
We will consider even Fourier modes in $\theta.$
\begin{lemma}
Solve $\Delta \psi =\omega$ on $\mathbb{R}^2.$
Then, we have that
\begin{align}\label{eq:psik}
    \psi_0(r)&=\int_0^r s^{-1}\int_0^s \tau \omega_0 (\tau) \, d\tau;\notag\\
    \psi_{2k}(r)&=r^{2k}\int_{r}^\infty s^{-4k-1}\int_0^s\tau^{2k+1}\omega_{2k}(\tau)\, d\tau, \quad k \ge 1.
\end{align}
\end{lemma}

Let us now focus our attention on \emph{non-negative} solutions that are even in $\theta$ and supported in the region \[\mathcal{C}:=\left\{(r,\theta): \frac{3\pi}{8}<|\theta|<\frac{5\pi}{8}\right\}.\]
A key property is that for $(r,\theta)\in\mathcal{C},$ we have that
\[|\cos(2k\theta)|\leq \sqrt{2}|\cos(2\theta)|\] for any $k.$ It follows that \[|\omega_{2k}(r)|\leq \sqrt{2}|\omega_2(r)|,\] for all $r\in [0,\infty)$ whenever $\omega$ is non-negative on $\mathbb{R}^2$ and supported in $\mathcal{C}.$
Now we have an important result.
\begin{lemma}\label{OscillationLemma}
Fix a non-negative $W\in C^4([\frac{3\pi}{8},\frac{5\pi}{8}]),$ vanishing to fourth order on the boundary. Then, there exist  universal constants $c,C>0$ so that 
\[\int_{\frac{3\pi}{8}}^{\frac{5\pi}{8}}R_1^2\omega W(\theta)\, d\theta\geq c\int_{r}^\infty\frac{|\omega_2(s)|}{s}\, ds-C\Big(|\omega_2(r)| +\frac{1}{r}\int_0^r |\omega_2(s)|\, ds+r\int_r^\infty \frac{|\omega_2(s)|}{s^2}\, ds\Big),\] for all $\omega$ non-negative and supported in $\mathcal{C}.$
\end{lemma}

\begin{proof}
To compute $R_1^2\omega,$ we write:
\[R_1^2\omega=-\partial_{xx}\sum_{k} \psi_{2k}(r)\cos(2k\theta),\]
where $\psi_{2k}$ is in \eqref{eq:psik}.
Now we simply use that 
\[\Big|\int_{\frac{3\pi}{8}}^{\frac{5\pi}{8}}\cos(2k\theta) W(\theta)d\theta\Big|\leq \frac{C}{1+k^4}\] and that $|\omega_{2k}|\leq \sqrt{2}|\omega_2|$ to estimate all terms with $k\not=1.$ For $k=1$, we get the most singular term only when both $x$ derivatives hit $r^2\cos(2\theta)$ and we must investigate
\[\mathcal{S}:=\int_r^\infty s^{-5}\int_0^s \tau^3\omega_{2}(\tau)\, d\tau.\] Upon integrating by parts, we see that 
\[\mathcal{S}=\frac{1}{4}\int_r^\infty \frac{\omega_2(s)}{s}ds+\frac{1}{4}r^{-4}\int_0^rs^3\omega_2(s)\, ds,\] the latter term can be estimated by:
\[r^{-4}\int_0^rs^3|\omega_2(s)|\, ds\leq \frac{1}{r}\int_0^r |\omega_2(s)|\, ds.\]
\end{proof}

\subsection{Jensen in $\theta$ and reduction to an equation on $\omega_2$}
Let us take the initial data 
\[\omega_0(r,\theta)=F_0(r)\Gamma(\theta),\] with $\Gamma\in C^\infty(\mathbb{S}^1)$ and supported in $\mathcal{C}$, even in $\theta$ and $\pi$-periodic. This simply means that the Fourier expansion in $\theta$ of $\Gamma$ (and thus $\omega(r,\cdot),$ by inspection) only contains terms of the form $\cos(2k\theta)$ for $k\in\mathbb{N}\cup \{0\}$. Note that, in order that $\omega_0$ be $C^\infty,$ we will need $F_0$ to vanish to infinite order at $r=0.$ This can be easily arranged. 
Now, formally solving the equation \eqref{eq:CSmodel}, we get that 
\[\omega=\omega_0\exp\left(\int_0^t R_1^2\omega\right).\] Note that $\omega_0$ being supported in $\mathcal{C}$ implies that $\omega$ is supported in $\mathcal{C}$ for all time. Let us now assume that \[\int_{\frac{3\pi}{8}}^{\frac{5\pi}{8}} \log (\Gamma (\theta)) \, d\theta=-M>-\infty.\]
Now, let us multiply by a weight $W$ as in Lemma \ref{OscillationLemma} only stipulating further that \[\int_{\frac{3\pi}{8}}^{\frac{5\pi}{8}}W(\theta) \, d\theta=1,\] and integrate in $\theta$ only:
\[\int_{\frac{3\pi}{8}}^{\frac{5\pi}{8}} \omega W \, d\theta= F_0\int_{\frac{3\pi}{8}}^{\frac{5\pi}{8}} \Gamma W \exp\left(\int_0^t R_{1}^2\omega\right) \,d\theta =F_0\int_{\frac{3\pi}{8}}^{\frac{5\pi}{8}} W \exp\Big(\log(\Gamma)+\int_0^t R_1^2\omega \Big) \, d\theta.\] 
Now we apply Jensen's inequality and deduce that 
\[\int_{\frac{3\pi}{8}}^{\frac{5\pi}{8}} \omega W\, d\theta \geq c(M,W) F_0\exp\Big(\int_0^t \int_{\frac{3\pi}{8}}^{\frac{5\pi}{8}} R_1^2\omega W \, d\theta \, ds\Big).\]
Next, we invoke Lemma \ref{OscillationLemma} and we deduce:
\[\int_{\frac{3\pi}{8}}^{\frac{5\pi}{8}} \omega W\, d\theta \geq c(M,W) F_0\exp\Big(\int_0^t \left[c\int_{r}^\infty\frac{|\omega_2(s)|}{s}ds-C\Big(|\omega_2(r)| +\frac{1}{r}\int_0^r |\omega_2(s)|ds+r\int_r^\infty \frac{|\omega_2(s)|}{s^2}\, ds\Big)\right] d\tau\Big).\]
Observing that 
\[\int_{\frac{3\pi}{8}}^{\frac{5\pi}{8}} \omega W \, d\theta \leq C|\omega_2(r)|,\] we deduce (upon calling $|\omega_2(t,r)|=f(t, r)$),
\begin{equation}\label{KeyLowerBound}f(r,t)\geq c F_0(r) \exp\Big(\int_0^t \left[c\int_{r}^\infty\frac{f(s, \tau)}{s}\, ds-C\Big(f(r,\tau)+\frac{1}{r}\int_0^r f(s, \tau)\, ds+r\int_r^\infty \frac{f(s,\tau)}{s^2}\, ds\Big)\right]d\tau\Big).\end{equation}
To simplify the notation, let us define the linear operator $L:$
\[L(g):=c\int_{r}^\infty\frac{g(s)}{s}\, ds-C\Big(g(r)+\frac{1}{r}\int_0^r g(s)\, ds+r\int_r^\infty \frac{g(s)}{s^2}\, ds\Big).\]
\subsection{$C^\alpha$ argument for $C^\infty$ solutions}
Now we make a choice of a weight in $(r, \theta)$ that will allow us to deduce singularity formation as before. Interestingly, the radial part of the weight that we choose is inspired by the $C^\alpha$ singularities constructed in \cite{E_Classical} even though we construct here smooth solutions that develop a singularity.
\begin{lemma}
For any $c,C>0,$ there exists a weight $W_2$ that is positive and integrable on $[0,\infty)$ for which 
\[W_1:=L^*(W_2)>0,\] where $L^*$ is the $L^2$ adjoint of $L.$
\end{lemma}
Let us start by observing that 
\[L^*(g)=\frac{1}{2r}\int_0^rg(s)\, ds-C\Big(g(r)+\int_r^\infty \frac{g(s)}{s}\, ds +\frac{1}{r^2}\int_0^r s g(s)\, ds\Big).\]
\begin{proof}
For $\alpha$ small,  take 
\[W_2(r)=\frac{r^{-1+\alpha}}{1+r^{2\alpha}}.\] In this case, the first term in $L^*(g)$ can be computed directly:
\[\frac{c}{r}\int_0^r W_2(s)\, ds=\frac{c}{r}\int_0^r\frac{s^{-1+\alpha}}{1+s^{2\alpha}}\, ds=\frac{c}{\alpha r}\int_0^{r^\alpha} \frac{ds}{1+s^2}=\frac{c\arctan(r^\alpha)}{\alpha r}.\] It is not difficult to show that for $\alpha$ sufficiently small, this term dominates all the other terms. 
In particular, 
\[W_1(r):=L^*(W_2)(r)>\frac{c\arctan(r^\alpha)}{2\alpha r}\] for all $r\geq 0.$ 
\end{proof}
Now let us establish singularity formation for \eqref{eq:CSmodel} by studying \eqref{KeyLowerBound}. First, let us choose $F_0\in C^\infty([0,\infty))$ vanishing to infinite order at $r=0$ with the property that 
\[\int \log \left(\frac{F_0W_1}{W_2}\right) W_2=-K>-\infty.\] The existence of such an $F_0$ is easy to see since $W_1$ and $W_2$ are positive everywhere and algebraic as $r\rightarrow 0$ and $r\rightarrow\infty$. Now, testing \eqref{KeyLowerBound} with $W_1,$ we see that 
\[\int f {W}_1\geq c\int F_0 \mathrm{W}_1 \exp\Big(\int_0^t L(f)\Big)=c\int \exp\Big(\log \left(\frac{F_0 W_1}{W_2}\right)+\int_0^t L(f)\Big) {W}_2.\] Applying Jensen's inequality and using the definition of ${W}_1$, we see that
\[\int f W_1\geq c\exp\Big(\int_0^t \int f W_1\Big).\] Singularity formation now follows. 

\subsection{The case of strictly positive solutions?}

The proof of Theorem \ref{EnergyTheorem} uses that the data is compactly supported (specifically that it is supported in the cone $\mathcal{C}$). It may be asked whether it is possible to get a singularity for \emph{positive} solutions, say, on $\mathbb{T}^2$.
Observe that, setting $\omega=\exp (f)$, positive solutions to \eqref{eq:CSmodel} are equivalent to solutions of:
\[\partial_t f = R_{1}^2(\exp(f)).\] Let us multiply this equation by $\Delta f.$
Then we see that
\[\frac{1}{2}\frac{d}{dt}|f|_{H^1}^2=-\int R_1^2(\exp(f))\Delta f=\int \partial_{xx} \exp(f) f=-\int \exp(f)(\partial_x f)^2.\] It follows that $|f|_{H^1}$ is non-increasing. It is not difficult to show that this implies that, for positive solutions, we have \emph{a-priori} control of all the $L^p$ norms of $\omega$ for $p<\infty$ (this follows from the Moser-Trudinger inequality applied to $f=\log \omega$). While this does not necessarily imply global regularity, as far as we can tell, it does indicate that there may be a serious difference between considering \emph{positive} solutions and non-negative solutions (see also Question 5 of \cite{drivas2023singularity}). Certainly, such bounds rule out the existence of positive self-similar blow-up solutions. 

\section{Some Concluding Remarks}

We gave a new technique to establish singularity formation in \emph{scalar} stretching problems. The technique was flexible enough to handle even equations with a dissipative structure and answer a few open problems. Two important directions for future consideration include the matrix version:
\begin{equation}\label{MatrixCase}\partial_t A=R(A)A,\end{equation} with $R$ some operator acting on matrices as well as the advective problem both in the scalar case and the matrix/vector case:
\begin{equation}\label{TransportCase}\partial_t f + u\cdot\nabla f = f R(f).\end{equation}
For the matrix case, \eqref{MatrixCase}, it is possible to consider the evolution of the determinant:
\[\frac{d}{dt}\text{det}(A)=tr(R(A))\text{det}(A).\] Using this idea, one can derive similar singularity results for certain classes of systems. Unfortunately, for most of the systems we are concerned with (such as the evolution of the gradient of the flow-map in incompressible flows), the determinant is conserved and so no singularity can be deduced so simply. Still, it is possible that there are other extensions of the ideas presented above that do not require us to reduce to a scalar problem. Investigating this question is of great interest. For the case with transport given in \eqref{TransportCase}, there are obvious extensions once some assumptions are made on the relationship between $R$ and $u$. It remains to be seen whether such ideas can be used to establish singularity formation in relevant physical systems like the incompressible Euler equation.

\section*{Acknowledgements}
R.B. acknowledges funding from the Italian Ministry of University and Research, project  PRIN 2022HSSYPN and from the Royal Society of London, International Exchange Grant 2020.
T.M.E. acknowledges funding from the NSF DMS-2043024, an Alfred P. Sloan Fellowship, and a Simons Fellowship. A significant portion of this work was carried out at Princeton University. The authors thank the Department of Mathematics there for their hospitality.

\bibliographystyle{siam}
\bibliography{biblio}

\begin{thebibliography}{1}

\bibitem{CLM}
{\sc P.~Constantin, P.~D. Lax, and A.~Majda}, {\em A simple one-dimensional
  model for the three-dimensional vorticity equation}, Comm. Pure Appl. Math.,
  38 (1985), pp.~715--724.

\bibitem{constantin2012}
{\sc P.~Constantin and W.~Sun}, {\em Remarks on {O}ldroyd-{B} and related
  complex fluid models}, Commun. Math. Sci., 10 (2012), pp.~33--73.

\bibitem{drivas2023singularity}
{\sc T.~D. Drivas and T.~M. Elgindi}, {\em Singularity formation in the
  incompressible {E}uler equation in finite and infinite time}, EMS Surveys in
  Mathematical Sciences, 10 (2023), pp.~1--100.

\bibitem{E_Classical}
{\sc T.~Elgindi}, {\em Finite-time singularity formation for {$C^{1,\alpha}$}
  solutions to the incompressible {E}uler equations on {$\Bbb R^3$}}, Ann. of
  Math. (2), 194 (2021), pp.~647--727.

\bibitem{feng2015cauchy}
{\sc H.~Feng and V.~{\v{S}}ver{\'a}k}, {\em On the {C}auchy problem for
  axi-symmetric vortex rings}, Archive for Rational Mechanics and Analysis, 215
  (2015), pp.~89--123.

\bibitem{GG}
{\sc N.~Garofalo and P.~B. Garrett}, {\em ${A}_p$-weight properties of real
  analytic functions in $\mathbb{R}^n$}, Proceedings of the American
  Mathematical Society, 96 (1986), pp.~636--642.

\bibitem{Kiselev}
{\sc L.~Grafakos, M.~Pramanik, A.~Seeger, B.~Stovall, et~al.}, {\em Some
  problems in harmonic analysis}, arXiv preprint arXiv:1701.06637,  (2017).

\bibitem{marzuola2013relaxation}
{\sc J.~L. Marzuola and J.~Weare}, {\em Relaxation of a family of broken-bond
  crystal-surface models}, Physical Review E, 88 (2013), p.~032403.

\bibitem{miller2023finite}
{\sc E.~Miller}, {\em Finite-time blowup for the inviscid vortex stretching
  equation}, Nonlinearity, 36 (2023), p.~4086.

\end{thebibliography}
\end{document}